\documentclass[12pt]{amsart}

\newcommand{\R}{\mathbb{R}}

\theoremstyle{plain}
\newtheorem{theorem}{Theorem}[section]

\newtheorem{proposition}[theorem]{Proposition}
\newtheorem{lemma}[theorem]{Lemma}
\newtheorem{corollary}[theorem]{Corollary}
\theoremstyle{remark}
\newtheorem{remark}[theorem]{Remark}

\theoremstyle{definition}
\newtheorem{definition}[theorem]{Definition}
\newtheorem{example}[theorem]{Example}
\newtheorem{problem}[theorem]{Problem}

\begin{document}

\baselineskip 16.5pt
\title[Convergence of Greedy algorithms for Haar System] {On the Convergence of Greedy Algorithms for Initial Segments of the Haar Basis}
\author[Dilworth]{S. J. Dilworth}
\address{Department of Mathematics\\ University of South Carolina\\
Columbia, SC 29208\\ U.S.A.} \email{dilworth@math.sc.edu}
\author[Odell]{E. Odell}
\address{Department of Mathematics \\
The University of Texas\\1 University Station C1200\\
Austin, TX 78712\\ U.S.A.}
\email{odell@math.utexas.edu}
\author[Schlumprecht]{Th. Schlumprecht}
\address{Department of Mathematics, Texas A \& M University\\
College Station, TX 78743, U.S.A.}
\email{schlump@math.tamu.edu}
\author[Zs\'ak]{Andr\'as Zs\'ak}
\address{Department of Mathematics and Statistics, Fylde College,  Lancaster university, Lancaster LA1 4YF, U.K. $\&$}
\address{Peterhouse,
Cambridge, CB2 1RD, U.K.}

\email{A.Zsak@dpmms.cam.ac.uk}
\thanks{\textit{2000 Mathematics Subject Classification}: Primary: 41A65. Secondary: 42A10, 46B20}
\keywords{greedy algorithms; Banach spaces; Haar basis.}
\thanks{The first three authors were supported by the NSF}
\maketitle
\begin{abstract} We consider the $X$-Greedy Algorithm and the Dual Greedy Algorithm in a finite-dimensional Banach space  with a strictly monotone basis
as the dictionary. We show that when the dictionary is an  initial segment of the Haar basis in  $L_p[0,1]$ ($1 < p < \infty$) then the algorithms terminate after
finitely many iterations and that the number of iterations is bounded by a function of the length of the initial segment. We also prove a more general result  for a class of  strictly monotone bases. \end{abstract}
\section{Introduction}\label{intro}

Greedy algorithms in Hilbert space are known to have  good convergence properties.  The first general result in this direction was obtained by Huber \cite{H},
who proved convergence of the \textit{Pure Greedy Algorithm} (PGA) in the weak topology of a Hilbert space $H$ and conjectured that the PGA  converges strongly  in $H$. Huber's  conjecture was proved by Jones \cite{J1}.

Our interest in this paper is in convergence results for greedy
algorithms in a Banach space $X$ (see \cite{T1}). We say that
$\mathcal{D} \subset X$ is  a \textit{dictionary}  if the linear
span of $\mathcal{D}$ is norm-dense in $X$ and $\|\varphi\|=1$ for
all $\varphi \in \mathcal{D}$. (Usually, but not here, $\mathcal{D}$ is  also assumed to be symmetric.) For some of the algorithms that
have been proposed, e.g.\ the \textit{Weak Chebyshev Dual Greedy
Algorithm} \cite{T2, DKT} or the \textit{Weak Greedy Algorithm
with Free Relaxation} \cite{T3}, it is known that  uniform
smoothness of $X$ guarantees strong convergence of these
algorithms for an arbitrary dictionary $\mathcal D$. Rate of
convergence results have also been proved \cite{T2, T3}.

We are mainly concerned with two natural generalizations of the
PGA to the Banach space setting, namely the \textit{$X$-Greedy
Algorithm} (XGA) and the \textit{Dual Greedy Algorithm} (DGA) (see \cite{T1}).
These algorithms generate a sequence of greedy approximants $(G_n)$
to an initial  vector $x$.  The updated approximant $G_{n+1}$ is
obtained from $G_n$ by  best  one-term approximation of  the
residual $x - G_n$ in the direction of  a particular  dictionary
element $\varphi_n \in \mathcal{D}$ which satisfies a certain
selection criterion. Precise definitions will be given below. 

Livshits \cite{L} constructed a dictionary in a smooth Banach space 
 for which  the XGA fails to converge.
No
general convergence results for the strong topology are known for
the XGA and the DGA for the class of  uniformly smooth Banach
spaces. In \cite{DKSTW} convergence was proved  (for an arbitrary
dictionary) for the weak topology in uniformly smooth Banach
spaces with the so-called \textit{$WN$ Property}. In particular,
weak convergence was proved in uniformly smooth Banach spaces
which are uniformly convex and have a $1$-unconditional basis.
Unfortunately, $L_p[0,1]$  ($p\ne2$) does not enjoy  the $WN$
Property, so these results cannot be applied to $L_p[0,1]$.

An important advance was made by Ganichev and  Kalton \cite{GK} who proved  strong convergence of the DGA in $L_p[0,1]$
for an arbitrary dictionary. More precisely, they
introduced a geometrical  property called Property $\Gamma$, proved strong convergence of the DGA in  Banach spaces with Property $\Gamma$,
and showed that all subspaces of quotient spaces of $L_p[0,1]$ ($1<p<\infty$) enjoy Property $\Gamma$. In \cite{GK2} property $\Gamma$ was characterized via the notion of a `tame' convex function, and using this characterization several other important  spaces  were shown to enjoy Property $\Gamma$.

The arguments used by Ganichev and Kalton  do not seem to yield convergence results for the XGA. In particular, convergence of the XGA in $L_p[0,1]$ is an open question. This is surprising because the XGA yields the best one-term approximation at each step.
Even  for the important special case of this problem in which  the dictionary is
the Haar basis of $L_p[0,1]$ very little seems to be known.

\begin{problem}  Suppose that the dictionary is the Haar basis in $L_p[0,1]$ ($p\ne2$). Does the XGA converge strongly to the initial vector $x$? Does it converge in the weak topology? 
\end{problem}

We attacked the
finite-dimensional analogue of this problem and obtained the
following theorem, which is a corollary of our main result
(Theorem~\ref{thm: maintheorem} below).

\begin{theorem}  \label{thm: introtheorem} Let $1 < p < \infty$ and let $(h^{(p)}_i)_{i=0}^\infty$ be the normalized Haar basis for $L_p[0,1]$. Then, for each $m\ge0$, there exists a positive integer $N(p,m)$ such that, for the dictionary $(h^{(p)}_i)_{i=0}^m$, the XGA and $DGA$ terminate in at most $N(p,m)$ iterations for every initial vector in the linear span of $(h^{(p)}_i)_{i=0}^m$. \end{theorem}

We present an example of a non-monotone basis of the two-dimensional Euclidean space for which the XGA does not terminate.
When the dictionary  is  a strictly monotone finite basis we show that for every initial vector the XGA and DGA terminate after finitely many iterations.
To
get a uniform bound on the number of iterations that is   independent of the initial vector, as in Theorem~\ref{thm: introtheorem}, we isolate a particular property (Property P) of the Haar basis and prove
the existence of a uniform bound for all strictly monotone bases with Property P.

The paper is organized as follows. The greedy algorithms which we consider are defined in the next section. Our main result  is proved in Section~\ref{sec: mainresults}. The final section contains two estimates for the Haar basis which lead to a refinement of Theorem~\ref{thm: introtheorem} in the range $p>2$.

\section{Definitions and Notation}\label{definitions}
First we recall some notation and terminology from Banach space
theory. We denote the unit sphere $\{ x \in X \colon \|x\|=1\}$ of
$X$ by $S_X$. We say that $F_x\in X^{\ast}$ is a \textit{norming
functional} for a nonzero $x\in X$ when $\|F_x\| _{X^{\ast}}= 1$
and $F_x(x) = \|x\|$; by the Hahn-Banach theorem, each $x\in X$
has at least one norming functional. $X$ is \textit{smooth} if
$F_x$ is unique. 

It is known that the norm of a  smooth finite-dimensional Banach space is \textit{uniformly Fr\'echet differentiable}, i.e.
\begin{equation} \label{eq: uniformlyfrechet}\|x + y\| = 1 + F_x(y) + \varepsilon(x,y)\|y\| \end{equation}
for all $x,y \in X$ with $\|x\|=1$, where $\varepsilon(x,y) \rightarrow 0$ uniformly for $(x,y) \in S_X \times X$  as $\|y\| \rightarrow 0$.

A basis $(e_i)_{i=1}^m$ of an $m$-dimensional Banach space $X$ is said to be \textit{strictly monotone} if
\begin{equation*} \|\sum_{i=1}^{i_0} a_i e_i \| \le \|\sum_{i=1}^m a_i e_i\| \end{equation*}
for all  $1 \le i_0 < m$ and $(a_i) \subset \mathbb{R}$ with
equality only if $a_i =0$ for $i=i_0+1,\dots,m$. The dual basis
$(e^*_i)_{i=1}^m \subset X^*$ is defined by $e_i^*(e_j) =
\delta_{i,j}$. The basis is \textit{normalized} if $\|e_i\|=1$ for
$i=1,\dots,m$. Note that  if $(e_i)_{i=1}^m$ is a  normalized
monotone basis then for all $(a_i) \subset \mathbb{R}$,
we have \begin{equation} \label{eq: monotonebasis} \frac{1}{2}
\max_{1 \le i \le m} |a_i| \le \|\sum_{i=1}^m a_i e_i\| \le
\sum_{i=1}^m |a_i|.\end{equation}

Let us recall the definition of the  Haar basis functions defined on $[0,1]$. Let $h_0 \equiv 1$.
For $n \ge 0$ and $0 \le k < 2^n$, we define $h_i$ for $i=2^n+k$ thus:
\begin{equation*}
h_i = \begin{cases} 1 &\text{on $[k/2^{n},(2k+1)/2^{n+1})$} \\
-1 &\text{on $[(2k+1)/2^{n+1},(k+1)/2^{n})$}\\
0 &\text{elsewhere}. \end{cases} \end{equation*} The Haar basis is a
strictly monotone basis of $L_p[0,1]$ (equipped with its usual norm $\|\cdot\|_p$) for $1<p<\infty$.

The algorithms which we consider in this paper all arise from the repeated application of a \textit{greedy step} to a nonzero \textit{residual} vector $y \in X$.
Let us describe the general form of this greedy step. \begin{itemize}
\item[(i)] Select  $\varphi(y) \in \mathcal{D}$ by applying a selection procedure (which depends on the  particular algorithm in question) to $y$.
In general the selection procedure will allow
many possible choices for   $\varphi(y)$.
\item[(ii)] Then select $\lambda(y) \in \mathbb{R}$ to minimize $\|y - \lambda \phi(y)\|$ over $\lambda$.
\end{itemize}

Starting with an \textit{initial vector} $x \in X$, we generate a sequence of residuals $(x_n)$  as follows.
\begin{itemize}

\item[(i)] Set $x_0 := x$.
\item[(ii)] For $n \ge 1$, apply the greedy step to the residual $y=x_{n-1}$ to obtain $\varphi_n := \varphi(x_{n-1}) \in \mathcal{D}$ and $\lambda_n := \lambda(x_{n-1}) \in \mathbb{R}$.
\item[(iii)] Set $x_n := x_{n-1} - \lambda_n \varphi_n$ to be the updated residual. \end{itemize}
The algorithm is said to \textit{converge} (strongly) if $\|x_n\|\rightarrow0$ as $n \rightarrow \infty$. It is said to \textit{terminate after $N$ steps} if $x_N =0$. For $n \ge 1$, the  $n^{th}$ \textit{greedy approximant} is defined by $G_n = \sum_{i=1}^n \lambda_i \varphi_i$. Note that $G_n = x - x_n$ and that $x = \sum_{i=1}^\infty \lambda_i \varphi_i$ (resp.\ $x = \sum_{i=1}^N \lambda_i \varphi_i$) if the algorithm converges (resp.\ terminates after $N$ steps).

Two important greedy algorithms of this type are the \textit{weak $X$-Greedy Algorithm} (WXGA) and the \textit{Weak Dual Greedy Algorithm} (WDGA) (see \cite{T1}). In both
cases a
\textit{weakness parameter} $\tau \in (0,1)$ is specified in advance. For the WXGA with weakness parameter $\tau$ the greedy step is as follows.
Given a nonzero $x \in X$, we select $\varphi \in \mathcal{D}$ to
satisfy
\begin{equation} \label{eq: WXGAgreedystep} \| x \| - \min_{\lambda \in \mathbb{R}}\| x - \lambda \varphi(x)\|
\geq
\tau\Bigl(\|x\| - \inf_{\substack{\lambda\in\R \\
\varphi\in \mathcal{D}}} \|x - \lambda\varphi\|\Bigr). \end{equation}
We can also set $\tau =1$ in the above when it can be shown that the infimum in \eqref{eq: WXGAgreedystep} is attained, e.g. if
  $\mathcal{D}$ is finite or if  $\mathcal{D}$ is a monotone basis for $X$; the case $\tau=1$ is  the \textit{$X$-Greedy Algorithm} (XGA)
discussed in the Introduction. 

For the WDGA with weakness parameter $\tau$ the greedy step is as follows. Given a nonzero $y \in X$, choose $\varphi(y) \in \mathcal{D}$ such that
\begin{equation*} |F_y(\varphi(y))| \ge \tau \sup_{\varphi \in \mathcal{D}} |F_y(\varphi)|. \end{equation*}
The case $\tau=1$, when it makes sense, is the \textit{Dual Greedy Algorithm} (DGA) discussed in the Introduction. Smoothness of $X$ guarantees that the residuals
satisfy $\|x_n\| < \|x_{n-1}\|$ for both the $WXGA$ and the
$WDGA$.

\section{Main Results} \label{sec: mainresults}
\begin{proposition} Suppose that $X$ is a finite-dimensional smooth Banach space.
Then there exists $\gamma \in (0,1)$ such that the
greedy steps of both the WXGA and WDGA applied to any nonzero $y \in X$ satisfy
\begin{equation} \label{eq: gamma} \|y - \lambda(y) \varphi(y)\| \le \gamma \|y\|. \end{equation} \end{proposition}
\begin{proof} First we consider the WDGA with weakness parameter $\tau$. By compactness of $S_X$ and continuity of the mapping $y \rightarrow F_y$,
there exists $\delta>0$ such that
$$ \sup \{|F_y(\phi)| \colon \phi \in \mathcal{D} \} \ge \delta \qquad (y \in S_X).$$
Hence, the WDGA applied  to $y \in S_X$ selects $\varphi(y)\in
\mathcal{D}$ such that $|F_y(\varphi(y))| \ge \tau\delta$. By
uniform Fr\'echet differentiability of the norm there exists
$\eta>0$ such that for all $y \in S_X$ and for all $z \in X$ with
$\|z\| \le \eta$, we have $|\varepsilon(y,z)| \le \tau\delta/2$ in \eqref{eq: uniformlyfrechet}, and hence \begin{align*}
 \|y - z\| &= 1-F_y(z) + \varepsilon(y,-z)\|z\|\\
&\le 1 - F_y(z) + \frac{\tau\delta}{2}\eta. \end{align*}
Setting $z = \pm \eta \varphi(y)$ for the appropriate choice of signs  yields $F_y(z) \ge \eta\tau\delta$, and hence
$$\|y-z\| \le 1-\frac{\eta\tau\delta}{2}.$$
 By homogeneity we get for all nonzero $y \in X$
\begin{equation} \label{eq: gammaestimate} \|y - \lambda(y) \varphi(y)\| \le  (1-\frac{\eta\tau\delta}{2})\|y\|. \end{equation}
Setting $\tau=1$ in the above yields an estimate for the DGA. Since the greedy step of the XGA produces a residual with the smallest norm, it follows that the same estimate must also hold for the XGA. But this implies that \eqref{eq: gammaestimate} also holds for the WXGA with parameter $\tau$.
\end{proof}

We turn now to consider the case in which $X$ is $m$-dimensional ($1 \le m < \infty$)  and the dictionary is a strictly monotone normalized basis
$B = (e_i)_{i=1}^m$ for $X$. We shall say that the algorithm is \textit{norm-reducing with constant $\gamma$} ($0 < \gamma < 1$) if
\eqref{eq: gamma} holds for the greedy step.
\begin{proposition} Suppose that the algorithm is norm-reducing with constant $\gamma$. Then, for each initial vector $x \in X$, the algorithm terminates after finitely many steps. \end{proposition}.
\begin{proof} The proof is by induction on $m$. The result is trivial if $m=1$, so suppose $m>1$ and that $x = \sum_{i=1}^m a_i e_i$. If $a_m =0$, then by monotonicity of $B$  the algorithm will never select $e_m$, so the result follows by induction. So suppose that $a_m \ne 0$. If
 the algorithm selects $e_m$ at the $n^{th}$ step, then by strict monotonicity the new residual $x_n$ satisfies $e_m^*(x_n) = 0$, i.e. the last coefficient is set equal to zero, and the result follows by induction. Thus to conclude the proof it suffices to show that $e_m$ is eventually selected.
But if $e_m$ is never selected then $e_m^*(x_n) = a_m$ for all $n
\ge 1$,  so by \eqref{eq: monotonebasis}
$$ \gamma^n \|x\| \ge \|x_n\| \ge \frac{1}{2} \max_{1 \le i \le m} |e_i^*(x_n)| \ge \frac{|a_m|}{2},$$
which is a contradiction when $n$ is larger than $\ln(2\|x\|/|a_m|)/\ln(\gamma^{-1})$.
\end{proof}

\begin{example} Monotonicity of the basis is essential. Indeed, consider the basis $B = \{ (1,0), (1/\sqrt 2, 1/\sqrt 2)\}$ of $2$-dimensional Euclidean space. It is easily seen that the XGA does not terminate unless the initial vector is a multiple of one of the basis vectors. \end{example}
\begin{problem} The estimate  $n \le \ln(2\|x\|/|a_m|)/\ln(\gamma^{-1})$ for the number of steps before the algorithm terminates clearly depends on $x$ and  becomes unbounded as $a_m \rightarrow 0$. Is there a uniform bound $N$ which is independent of the initial vector $x$? \end{problem}

We shall now provide a sufficient condition which guarantees a positive answer to this question. Then we verify that the initial segments of the Haar basis satisfy this condition.
\begin{definition} Let $B=(e_i)_{i=1}^m$ be a  normalized monotone basis for $X$. We say that $B$ has \textit{Property $P$ with constant $\zeta>0$} if the following
condition is satisfied: for all $x = \sum_{i=1}^m a_i e_i \in X$ and for all $1 \le i_0 \le m-1$, we have
\begin{equation*}
|t_0| \le \zeta \sum_{i=i_0+1}^m |a_i|,
\end{equation*} where $t_0$ minimizes the mapping $t \mapsto \|\sum_{i=1}^{i_0-1} a_ie_i + te_{i_0} + \sum_{i=i_0+1}^m a_i e_i\|$.
\end{definition}
Now we can state our main result.
\begin{theorem} \label{thm: maintheorem} Suppose that $X$ is $m$-dimensional, that $B$ is a strictly monotone basis for $X$ which has Property P with constant $\zeta$, and that the algorithm is norm-reducing with
constant $\gamma$. Then there exists a positive integer $N(m, \gamma, \zeta)$ such that the algorithm terminates in at most $N$ steps for every initial vector $x \in X$. \end{theorem}
The proof of Theorem~\ref{thm: maintheorem} requires some combinatorial notation which we shall now describe. For positive integers $r$ and $s$, with $r\le s$,  the integer interval  $\{n \in \mathbb{N} \colon r \le n \le s\}$ will be denoted by $[r,s]$. If $I_1$ and $I_2$ are integer intervals we write
$I_2<I_1$ if
$\max I_2 < \min I_1$, and we say they are consecutive if $\min I_1 = \max I_2 +1$.

For $1 \le k \le m$,  an \textit{interval partition} of $[1,m]$
is a $k$-tuple $P=(I_1,\dots,I_k)$ of consecutive integer intervals $I_1,\dots,I_k$ such that $\min I_k =1$, $\max I_1 =m$, and
$I_k < I_{k-1}<\dots<I_1$. The collection  $ \mathcal{P}(m)$ of all interval partitions of $[1,m]$ is readily seen to have cardinality $2^{m-1}$. We endow
$\mathcal{P}(m)$ with the \textit{lexicographical ordering} $\prec$, i.e., if $P_1 = (I_1,\dots,I_r)$ and $P_2 = (J_1,\dots, J_s)$ are two interval partitions then
$P_1 \prec P_2$ if, for some $t \ge 1$, we have $\operatorname{card} I_u = \operatorname{card} J_u$ for $1 \le u < t$ and
$\operatorname{card} I_t < \operatorname{card} J_t$. Note that $([1,m])$ is the maximum element of $(\mathcal{P}(m),\prec)$.

Next we associate to each $y= \sum_{i=1}^m a_i e_i  \in X$ an interval partition $P(y) = (I_1,\dots,I_k) \in \mathcal{P}(m)$ by `backwards induction'
as follows: \begin{itemize}
\item[(i)] $m \in I_1$; \\
\item[(ii)] Suppose that $1 \le i < m$ and that $i+1 \in I_j$. Then
\begin{equation} \label{eq: intervalpartition} i \in \begin{cases} I_j& \text{if $|a_i| \le (1+ \zeta)^{m-i} \sum_{r=1}^j |a_{\max I_r}|$},\\
I_{j+1}\, &\text{otherwise}. \end{cases} \end{equation}
\end{itemize}
It may be helpful to  explain the  intuition behind this definition. The definition of  $P(y)$ begins with $I_1$. Working backwards from  $i=m \in I_1$, then $i$ is placed in the same interval $I_j$ as $i+1$ if the coefficient $|a_i|$ is  not too much larger (roughly speaking) than the later coefficients $|a_{i+1}|,\dots,|a_m|$. But if $|a_i|$ is  much larger than  the later coefficients  then
a new interval $I_{j+1}$ is begun for which $i = \max I_{j+1}$.
Note that \begin{equation}\label{eq: normofyupperbound} \begin{split}
\|y\| &\le \sum_{i=1}^m |a_i|\\&= \sum_{j=1}^k \sum_{i\in I_j} |a_i| \\  
 &\le  ( \sum_{j=1}^k |a_{\max I_j}|) \sum_{i=1}^m (1+ \zeta)^{m-i}\\
&\le m \frac{(1+\zeta)^{m}}{\zeta} \max_{1\le j \le k}|a_{\max I_j}|. \end{split} \end{equation}

\begin{lemma} \label{lem: n_0estimate}For each initial vector $y \in X$ 
with $P(y) = (I_1,\dots,I_k)$ there exists $i_0 \in \{\max I_j \colon 1 \le j \le k\}$ such that 
the algorithm selects $e_{i_0}$ in at most
$n_0$ steps, where\begin{equation} \label{eq: numberofsteps}
n_0 \le  1+ \lfloor \frac{\ln(2m(1+\zeta)^{m}/\zeta)}{\ln(1/\gamma)} \rfloor. \end{equation} \end{lemma}

\begin{proof} Let $i_0$  be defined by
\begin{equation*} |a_{i_0}| = \max \{|a_i| \colon i \in  \{\max I_j \colon 1 \le j \le k\}\}. \end{equation*}
Suppose that  $e_{i_0}$ is first selected at the $(n_0)^{th}$ step. Then  the residual $y_{n_0-1}$ satisfies by 
\eqref{eq: monotonebasis} and \eqref{eq: normofyupperbound}
\begin{equation*}
\frac{|a_{i_0}|}{2} \le \|y_{n_0-1}\| \le \gamma^{n_0-1}\|y\| \le
\gamma^{n_0-1}m\frac{(1+\zeta)^m}{\zeta}|a_{i_0}|,
\end{equation*} and the result follows.
\end{proof}

\begin{lemma} \label{lem: lexordering}  Suppose that when applied to $y$ the algorithm  selects $e_{i_0}$ and produces a residual $z$. Let $P(y) = (I_1,\dots,I_k)$
and $P(z) = (J_1,\dots, J_l)$.  Then either $i_0 = m$ or  \begin{equation*}
P(y) \begin{cases} \prec P(z)& \text{if $i_0 \in \{\max I_j \colon 2 \le j \le k\}$},\\
= P(z) &\text{otherwise}. \end{cases} \end{equation*}
\end{lemma} \begin{proof} We may assume that $i_0 < m$. Suppose that $i_0+1 \in J_{j_0}$.
Let $y = \sum_{i=1}^m a_i e_i$ and $z = \sum_{i=1}^m b_i e_i$. Clearly, $b_i = a_i$ if $i \ne i_0$. Thus by \eqref{eq: intervalpartition}, $J_j = I_j$ for $j < j_0$ and
$\max J_{j_0} = \max I_{j_0}$. Since $B$ has Property P with constant $\zeta$, and using the estimate $|a_i| \le (1+\zeta)^{m-i} (\sum_{j=1}^{j_0} |a_{\max I_j}|)$ for $i > i_0$ which follows from \eqref{eq: intervalpartition}, we get
\begin{equation*} \begin{split}
|b_{i_0}| &\le \zeta \sum_{i=i_0+1}^m |a_i| \\
&\le  \zeta(\sum_{i=i_0+1}^m (1+\zeta)^{m-i}) (\sum_{j=1}^{j_0} |a_{\max I_j}|)\\
&\le (1+ \zeta)^{m-i_0} (\sum_{j=1}^{j_0} |b_{\max J_j}|).
\end{split} \end{equation*}
Thus, by \eqref{eq: intervalpartition}, $i_0 \in J_{j_0}$. In particular, if $i_0 \notin I_{j_0}$ (in which case $i_0=\max I_{j_0+1}$), then $\operatorname{card}(J_{j_0}) >  \operatorname{card}(I_{j_0})$, so $P(y) \prec P(z)$. On the other hand, if $i_0 \in I_{j_0}$, then using the facts that
$b_i = a_i$ if $i \ne i_0$ and that $i_0 \ne \max J_{j_0}$, it follows again from \eqref{eq: intervalpartition} that $P(y) = P(z)$.
\end{proof}

\begin{proof}[Proof of Theorem~\ref{thm: maintheorem}] The proof is by induction on $m$.  Let $x \in X$. We may assume that $e_m^*(x) \ne 0$.
It suffices to give a bound independent of $x$ for the number of
steps required for the algorithm to select $e_m$.  Let $P(x) =
(I_1,\dots,I_k)$. Then by Lemma~\ref{lem: n_0estimate} the
algorithm  selects either $e_m$ or  $e_{i_0}$, where $i_0 \in
\{\max I_j \colon 2 \le j \le k\}$, in at most $n_0$ steps. In the
latter case, by Lemma~\ref{lem: lexordering},  $P(x)\prec
P(x_{n_0})$. Repeating the argument with $x$ replaced by
$x_{n_0}$, we find that either $e_m$ is selected in the first
$2n_0$ steps or $P(x_{n_0}) \prec P(x_{2n_0})$.  After a total of
at most $\operatorname{card}(\mathcal{P}(m))-1=2^{m-1}-1$
iterations of this argument, we find that either $e_m$ is selected
in the first $(2^{m-1}-1)n_0$ steps or $P(x_{(2^{m-1}-1)n_0}) =
([1,m])$, the maximum element of $\mathcal{P}(m)$. In the latter
case, by Lemma~\ref{lem: n_0estimate}, $e_m$ will be selected in
at most a further $n_0$ steps. In conclusion, $e_m$ will be
selected in at most $2^{m-1}n_0$ steps. This leads to the estimate
\begin{equation} \label{eq: totalnumberofsteps} N(m,\gamma,\zeta) = n_0 \sum_{i=1}^m 2^{i-1} = (2^{m}-1)n_0. \end{equation}
\end{proof}

Our next goal is to show that  all initial segments of the Haar basis for $L_p[0,1]$ ($1<p<\infty$) have property P with constant $\zeta$ depending on
$m$ and $p$. In the next section we prove that if $p>2$ then $\zeta$  may be chosen  independently of $m$.
\begin{lemma} \label{lem: haarestimate}  Let $1<p<\infty$ and let $h^{(p)}_i = h_i/\|h_i\|_p$ ($i \ge 0$). For each $m\ge1$ there exists a positive constant $C(m,p)$ such that, for all $M \in \mathbb{R}$, if
$|a_1| \ge C(m,p) \sum_{j=2}^m |a_j|$, then \begin{equation} \label{eq: haarestimateinLp}
\|M+\sum_{i=1}^m a_i h^{(p)}_i\|_p \ge \|M+\sum_{i=2}^m a_i h^{(p)}_i\|_p. \end{equation} \end{lemma}
\begin{proof} If $M=0$ we can take $C(m,p)=2$ by an easy triangle inequality calculation. If $M \ne 0$ then by homogeneity of the norm we may assume that $M=1$.
By expanding in a Taylor series, we see that there exist positive constants $b_1,\dots,b_m$ such that
\begin{align*}
\|1+ \sum_{i=1}^m a_i h^{(p)}_i\|_p^p &= \int_0^1 |1+\sum_{i=1}^m a_i h^{(p)}_i|^p \, dt\\
&= 1 + \sum_{i=1}^m b_i a_i^2 + o(\sum_{i=1}^m a_i^2). \end{align*} Thus there exists $0<\varepsilon<1$ such that if $|a_1| = \sum_{i=2}^m |a_i| < \varepsilon$ then \eqref{eq: haarestimateinLp} is satisfied. By convexity of the mapping $$t \mapsto \|1+ th^{(p)}_1+ \sum_{i=2}^m a_i h^{(p)}_i\|_p,$$
it follows that  \eqref{eq: haarestimateinLp} is also satisfied whenever  $\sum_{i=2}^m |a_i| < \varepsilon$
and  $|a_1| \ge\sum_{i=2}^m |a_i|$.
Now suppose that $\sum_{i=2}^m |a_i| \ge \varepsilon$. If 
$$|a_1| \ge (2 + 2/\varepsilon)\sum_{i=2}^m |a_i|\ge 2 + 2\sum_{i=2}^m |a_i|,$$
 then by the triangle inequality
\begin{align*}
\|1+\sum_{i=1}^m a_i h^{(p)}_i\|_p &\ge |a_1| - 1 - \sum_{i=2}^m |a_i|\\
&\ge 2 + 2\sum_{i=2}^m |a_i| -1 -\sum_{i=2}^m |a_i|\\
&=1 +\sum_{i=2}^m |a_i|\\
&\ge \|1+\sum_{i=2}^m a_i h^{(p)}_i\|_p. \end{align*}
Thus, $C(m,p)= 2 + 2/\varepsilon$ works.
\end{proof}

\begin{proposition} \label{prop: propertyPforHaar} Let $1 < p <\infty$. For each $m \ge 1$, the initial segment $(h^{(p)}_i)_{i=0}^m$ of the Haar basis for $L_p[0,1]$ has property
P with constant $\zeta = C(m,p)$. \end{proposition}
\begin{proof} Let $0 \le i_0 < m$. Suppose  $t_0$ minimizes the function
$$t \mapsto \|\sum_{i=0}^{i_0-1} a_i h^{(p)}_i + th_{i_0}^p + \sum_{i=i_0+1}^ma_i h^{(p)}_i\|$$
for fixed coefficients $(a_i) \subset \mathbb{R}$. Suppose that $h_{i_0}$ is supported on the dyadic interval $I$ and let $M$ be the (constant) value
assumed by $\sum_{i=0}^{i_0-1} a_i h^{(p)}_i$ on $I$. Then $t_0$ minimizes the function
$$t \mapsto \int_I |M +t h^{(p)}_{i_0} + \sum_{i=i_0+1}^ma_i h^{(p)}_i|^p \, dx.$$  Lemma~\ref{lem: haarestimate} obviously transfers from $[0,1]$ to $I$.
So  $$|t_0| \le C(p,m) \sum_{i_0+1}^m |a_i|.$$
\end{proof}

Note that in view of the preceding result the initial segments of the Haar basis in $L_p[0,1]$ satisfy the hypotheses of Theorem~\ref{thm: maintheorem}.
Thus, Theorem~\ref{thm: introtheorem} is  a special case of Theorem~\ref{thm: maintheorem}.
\section{Further Results}

In this section we present some more precise estimates  for the Haar basis. First we estimate the norm-reducing constant $\gamma$.
Then we show that for $p>2$ the constant $\zeta$  for Property $P$ may be chosen to be independent of $m$.

Recall that the  \textit{modulus of smoothness} $\rho_X(t)$  of a Banach space $X$ is defined  for $0<t\le1$ by
\[ \rho_{X}(t) = \sup \Bigl\{ \frac{\|x+y\| + \|x-y\|}{2} -1 : x,y\in X, \|x\|=1, \|y\|=t\Bigr\}\] (see \cite[p. 59]{LT}).  The modulus of smoothness
for $L_p[0,1]$ satisfies
\begin{equation*} \rho_{L_p[0,1]}(t) \le \begin{cases} c_p t^p &\text{if $1<p\le2$},\\
c_pt^2&\text{if $2 \le p < \infty$},\end{cases} \end{equation*}
where $c_p$ is a constant (see \cite[p. 63]{LT}).
\begin{proposition} \label{prop: gammaest} Suppose that $m\ge1$ and that $A \subseteq \mathbb{N}$ has cardinality $m$. For $\mathcal{D}_A := (h^{(p)}_i)_{i\in A}$ and
$X_A := \operatorname{span} \mathcal{D}_A \subset L_p[0,1]$ we have that the DGA and XGA are norm-reducing with constant
\begin{equation*} \gamma \le \begin{cases} 1- c'_p m^{p/(2-2p)} &\text{if $1 < p \le 2$},\\
1-c'_p m^{(2-2p)/p} &\text{if $2 < p < \infty$}, \end{cases} \end{equation*}
where $c'_p$ is a constant depending only on $p$.\end{proposition}

\begin{proof} The XGA produces the greatest norm reduction at each step, so it suffices to prove the result for the DGA. For convenience let $c$  denote a constant depending only on $p$ whose precise value may change from line to line. First we consider the case $1  < p  \le 2$. Let $y = \sum_{i\in A} a_i h^{(p)}_i \in S_{X_A}$ and let $F_y = \sum_{i\in A} b_i h^{(q)}_i \in  S_{X_A^*}$, where $q = p/(p-1)$. Note that
$$\|F_y\|_q \ge \|F_y\|_{X_A^*} =1.$$
The Haar basis in $L_q[0,1]$ satisfies an upper $2$-estimate for $q>2$ (see \cite{AO}). Thus, $\sum_{i\in A} |b_i|^2 \ge c$, and since $\operatorname{card} A =m$ we get
\begin{equation*}
|b_{i_0}| := \max_{i \in A} |b_i| \ge  \frac{c}{\sqrt m}. \end{equation*} We may assume that $b_{i_0}>0$. Thus, for $t \ge 0$, we have
\begin{equation*}
\|y + th^{(p)}_{i_0}\|_p \ge F_y(y + th^{(p)}_{i_0}) = 1 + tb_{i_0} \ge 1 + \frac{ct}{\sqrt m}. \end{equation*} Hence
\begin{align*}
\|y - th^{(p)}_{i_0}\|_p &\le 2 - \|y + th^{(p)}_{i_0}\|_p + 2\rho_{L_p[0,1]}(t)\\
& \le 2 -  (1+\frac{ct}{\sqrt m}) + 2c_p t^p\\
&= 1 -\frac{ct}{\sqrt m}+ 2c_p t^p.
\end{align*} Choosing $t$ to minimize $1 - (ct/\sqrt m)+ 2c_p t^{p}$ yields $\gamma \le 1 - cm^{p/(2-2p)}$.  The case $p>2$  is proved similarly using the fact that the
Haar basis in $L_q[0,1]$ satisfies an upper $q$-estimate for $q<2$.
\end{proof}
\begin{proposition} Suppose that $2<p<\infty$. Then
for   all $ y \in \operatorname{span} (h_i)_{i=2}^\infty$, we have
\begin{equation*} \|1 + t\|y\|_ph_1 + y\|_p \ge \|1+ y\|_p \end{equation*}
provided $|t| \ge \max(4,2^{(p-3)/2}\sqrt{p(p-1)})$.
\end{proposition}
\begin{proof} If $\|y\|_p >1$ then the result holds for $|t| \ge 4$ by the triangle inequality. So assume $\|y\|_p \le 1$. For $p \ge 2$, $f(x) = |x|^p$ is twice differentiable. Thus, by the  Mean Value Theorem, for all $x \in \mathbb{R}$ there exists $0< \theta(x) < 1$ such that
\begin{equation*}
|1+x|^p = 1 + px +\frac{p(p-1)}{2}x^2|1+ \theta(x)x|^{p-2}. \end{equation*} Thus, for all  $ y \in \operatorname{span} (h_i)_{i=2}^\infty$ with $\|y\|_p \le1$, we have
\begin{align*}
\int_{0}^1 |1+ y(s)|^p\,ds &\le 1+ p\int_0^1 y(s)\,ds + \frac{p(p-1)}{2} \int_0^1 y(s)^2 |1+ |y(s)||^{p-2}\,ds\\
&= 1+  0 + \frac{p(p-1)}{2} \int_0^1 y(s)^2 |1+ |y(s)||^{p-2}\,ds\\
&\le 1 + \frac{p(p-1)}{2} \|y\|_p^2\, \|1+|y|\|_p^{p-2}\\
\intertext{(by H\"older's inequality for the conjugate indices $p/2$ and $p/(p-2)$)}
&\le 1 + 2^{p-2}(\frac{p(p-1)}{2})\|y\|_p^2,
\end{align*} using the fact that $\|y\|_p \le 1$ in the last line. Hence  \begin{equation} \label{eq: upperestimate}
\|1+y\|_p \le (1 + 2^{p-3}p(p-1)\|y\|_p^2)^{1/p}. \end{equation}
On the other hand, since $p>2$, we have
\begin{equation} \label{eq: lowerestimate} \begin{split}
\|1+t\|y\|_p h^{(p)}_1 + y\|_p &\ge \|1+t\|y\|_p h_1^{(p)} + y\|_2\\
&\ge \|1+t\|y\|_p h^{(p)}_1 \|_2\\ &= (1+ t^2 \|y\|_p^2)^{1/2}. \end{split} \end{equation}
Combining \eqref{eq: upperestimate} and \eqref{eq: lowerestimate} yields the result.
\end{proof}
\begin{corollary} \label{cor: independentofm} Let $2 < p <\infty$. Every finite subsequence of the Haar basis for $L_p[0,1]$ has property
P with constant $$\zeta = \max(4,2^{(p-3)/2}\sqrt{p(p-1)}).$$ \end{corollary}
Combining Proposition~\ref{prop: gammaest} with Corollary~\ref{cor: independentofm}, and using the estimates \eqref{eq: numberofsteps} and 
\eqref{eq: totalnumberofsteps} for the number of iterations, yields the following strengthening of Theorem~\ref{thm: introtheorem}
in the range $p>2$ in which the initial segment of the Haar basis of length $m$ is replaced by any subset of cardinality $m$.
\begin{theorem} Let $2<p<\infty$ and let $m \ge 1$. Then, for all $A \subset \mathbb{N}$ of cardinality $m$, the XGA and DGA terminate in at most 
$O(2^mm\ln m)$ iterations for the dictionary $\mathcal{D}_A$ and for every initial vector in $X_A$. \end{theorem}
\begin{remark} We do not know whether or not the last result holds also for $1<p<2$. \end{remark}

\end{document}